\documentclass{article}
\usepackage{amsmath, amssymb,latexsym}
\usepackage{amsthm}
\usepackage{multirow,bigdelim}
\usepackage{geometry}
\usepackage[dvipdfmx]{graphicx}
\usepackage{multicol, xcolor}
\usepackage{indentfirst}
\usepackage{tikz}
\usetikzlibrary{intersections, calc, arrows}
\usepackage{here}

\theoremstyle{definition}
\newtheorem{thm}{Theorem}[section]
\newtheorem*{thm*}{Theorem}

\newtheorem{lemm}[thm]{Lemma}
\newtheorem{prop}[thm]{Proposition}

\allowdisplaybreaks[4]

\begin{document}

\title{An example of non-compact totally complex submanifolds of compact quaternionic K\"{a}hler symmetric spaces}
\author{Yuuki Sasaki}
\date{}
\maketitle

\begin{abstract}

\rm{
Totally complex submanifolds of a quaternionic K\"{a}hler manifold are analogous to complex submanifolds of a K\"{a}hler manifold.
In this paper, we construct an example of a non-compact totally complex submanifold of maximal dimension of a compact quaternionic K\"{a}hler symmetric space, except for  quaternionic projective spaces.
A compact Lie group acts on our example isometrically, and this action is of cohomogeneity one.
Our example is a holomorphic line bundle over some Hermitian symmetric space of compact type. 
Moreover, each fiber is a totally geodesic submanifold of the ambient quaternionic K\"{a}hler symmetric space and our example is a ruled submanifold.
Our construction relies on the action of a subgroup of the isometry group and a maximal totally geodesic sphere with maximal sectional curvature known as a Helgason sphere.
Furthermore, we prove that there exist no compact submanifolds of the same dimension that contain our example as an open part, except where the ambient quaternionic K\"{a}hler symmetric space is a complex Grassmannian.
}

\end{abstract}


\section{Introduction}

A quaternionic K\"{a}hler manifold is defined as a Riemannian manifold whose holonomy group is a subgroup of $Sp(1) \cdot Sp(n)$, and many mathematicians have investigated quaternionic K\"{a}hler manifolds.
It is well known that any quaternionic K\"{a}hler manifold of dimension greater than $8$ is Einstein.
For any $4n$-dimensional $(n \geq 2)$ compact quaternionic K\"{a}hler manifold $M$ with positive scalar curvature, there exists on its twistor space $Z$, which is a K\"{a}hler Einstein $S^{2}$-bundle over $M$ and the projection onto $M$ is a Riemannian submersion with totally geodesic fibers \cite{Salamon}. 
Wolf classified the quaternionic K\"{a}hler manifolds with nonvanishing scalar curvature that are symmetric spaces \cite{Wolf}.
These symmetric spaces are called quaternionic K\"{a}hler symmetric spaces or Wolf spaces.
The only known examples of compact quaternionic K\"{a}hler manifolds with positive scalar curvature are the compact quaternionic K\"{a}hler symmetric spaces.

In this paper, we study totally complex submanifolds of a quaternionic K\"{a}hler manifold, which are analogous to complex submanifolds of a K\"{a}hler manifold.
Totally complex submanifolds were introduced by Funabashi \cite{Funabashi} and are defined as submanifolds endowed with an almost complex structure induced by the quaternionic K\"{a}hler structure.
Hence, any totally complex submanifold is an almost complex manifold.
Alekseevsky and Marchiafava proved that any totally complex submanifold of a quaternionic K\"{a}hler manifold with nonvanishing scalar curvature is a K\"{a}hler manifold with respect to the induced metric \cite{Alekseevsky-Marchiafava}.
Moreover, they showed that totally complex submanifolds are minimal \cite{Alekseevsky-Marchiafava2}.
It is known that there exists a 3-Sasakian $SO(3)$-bundle $S$ over a compact quaternionic K\"{a}hler manifold $M$ with positive scalar curvature.
The $SO(3)$-bundle $S$ is called the Konishi bundle \cite{Konishi}.
The metric cone $C(S)$ of $S$ is a hyperk\"{a}hler manifold.
Recently, Aslan, Karigiannis, and Madnick discovered a relationship among totally complex submanifolds of a quaternionic K\"{a}hler manifold $M$, complex Legendrian submanifolds of its twistor space $Z$, Legendrian submanifolds of the Konishi bundle $S$, and complex Lagrangian submanifolds of the metric cone $C(S)$ \cite{Aslan-Karigiannis-Madnick}.
Thus, totally complex submanifolds are closely related to important submanifolds of various manifolds, and constructing examples of totally complex submanifolds is an interesting problem.

The construction and classification of totally complex submanifolds of compact quaternionic K\"{a}hler symmetric spaces have been investigated extensively.
For example, Takeuchi classified totally geodesic totally complex submanifolds of maximal dimension of compact quaternionic K\"{a}hler symmetric spaces \cite{Takeuchi}.
Tsukada classified parallel totally complex submanifolds of maximal dimension of quaternionic projective spaces \cite{Tsukada}.
Moreover, Bedulli, Gori, and Podest\'{a} proved that compact homogeneous totally complex submanifolds of maximal dimension of quaternionic projective spaces are parallel \cite{Bedulli-Gori-Podesta}.
Thus, compact homogeneous totally complex submanifolds of maximal dimension of quaternionic projective spaces have been completely classified.
Kimura studied totally complex submanifolds of a complex 2-plane Grassmannian.
In particular, he discovered a relationship between Hopf real hypersurfaces of a complex projective space and totally complex submanifolds of maximal dimension of a complex 2-plane Grassmannian \cite{Kimura} \cite{Kimura2}.
Moreover, Tsukada classified compact homogeneous totally complex submanifolds of maximal dimension of a complex 2-plane Grassmannian \cite{Tsukada3}.
Enoyoshi and Tsukada constructed an example of a compact totally complex submanifold of maximal dimension of the associative Grassmannian $G_{2}/SO(4)$ \cite{Enoyoshi-Tsukada}.
However, these examples of totally complex submanifolds are all compact.
There are few examples of non-compact totally complex submanifolds aside from open parts of compact ones.

In the present paper, we construct an example of a non-compact totally complex submanifold of maximal dimension of a compact quaternionic K\"{a}hler symmetric space other than quaternionic projective spaces.
A compact Lie group acts on our example isometrically, and this action is of cohomogeneity one.
Moreover, our example is a holomorphic line bundle over a Hermitian symmetric space of compact type.
Each fiber is a totally geodesic submanifold of the ambient space and our example is a ruled submanifold.
Also, we show that our example is not contained in any compact submanifold of the same dimension as an open part, except where the ambient compact quaternionic K\"{a}hler symmetric space is a complex Grassmannian of rank 2.
In Section 2, we recall some definitions concerning quaternionic K\"{a}hler manifolds and totally complex submanifolds.
Moreover, we review the construction of compact quaternionic K\"{a}hler symmetric spaces as given in \cite{Wolf}.
In Section 3, we construct our example and prove that it is a totally complex submanifold of maximal dimension .
We also study whether there exists a compact submanifold of the same dimension that contains our example as an open part.

\section{Preliminaries}

Let $(M,g)$ be a $4n$-dimensional Riemannian manifold and $Q$ a rank $3$ subbundle of the endomorphism bundle $\mathrm{End}TM$.
We call $(M,Q,g)$ a quaternionic K\"{a}hler manifold if it satisfies the following conditions:

\begin{itemize}

\item[(a)]
For any $p \in M$, there exists a local frame $\{I, J, K \}$ of $Q$ defined on a neighborhood of $p$ such that
\[
\begin{split}
& I^{2} = J^{2} = K^{2} = -\mathrm{id}, \\
& IJ = -JI = K, \ \ JK = -KJ = I, \ \ KI = -IK = J. \\
\end{split}
\]

\item[(b)]
For any $p \in M, A \in Q_{p}$, and $X,Y \in T_{p}M$, $g(AX,Y) + g(X,AY) = 0$.

\item[(c)]
The vector bundle $Q$ is parallel with respect to the Levi-Civita connection of $g$.

\end{itemize}

It is well known that a quaternionic K\"{a}hler manifold is Einstein if $n \geq 2$.
In this paper, we assume that $\dim M \geq 8$.
Set $Z = \{ I \in Q \ ;\ I^{2} = -\mathrm{id} \}$.
Then, $Z$ is an $S^{2}$-bundle over $M$ and called the {\it twistor space} of $M$.

We recall the definition of totally complex submanifolds of quaternionic K\"{a}hler manifolds.
Let $N$ be a submanifold of $M$.
If there exist a section $I \in \Gamma(Z|_{N})$ such that $I(TN) \subset TN$, then $N$ is called an {\it almost complex submanifold} \cite{Alekseevsky-Marchiafava}.
Then, $(N,I)$ with the induced metric by $g$ is almost Hermitian.
If $I$ is integrable on $N$, then $N$ is called a {\it complex submanifold} \cite{Alekseevsky-Marchiafava}.
Moreover, if $I$ is parallel on $N$ with respect to the Levi-Civita connection of the induced metric, then $N$ is called a {\it K\"{a}hler submanifold} \cite{Alekseevsky-Marchiafava}.
For an almost complex submanifold $(N,I)$, set $Z_{I} := \{ J \in Z|_{N} \ ;\ IJ = -JI \}$.
If $J(TN) \perp TN$ for any $J \in Z_{I}$, then $N$ is called a {\it totally complex submanifold} \cite{Funabashi}.
If the scalar curvature of $M$ is nonzero, then an almost complex submanifold $N$ is K\"{a}hler if and only if it is totally complex \cite{Alekseevsky-Marchiafava}.
It is easy to see that $2 \dim N \leq \dim M$ for any totally complex submanifold $N$.
If $2\dim N = \dim M$, then $N$ is called a totally complex submanifold of maximal dimension.

Next, we recall the construction of a compact quaternionic K\"{a}hler symmetric space by Wolf \cite{Wolf}.
Let $G$ be a simply connected simple compact Lie group and $\frak{g}$ be the Lie algebra of $G$.
Choose a maximal abelian subspace $\frak{a}$ of $\frak{g}$.
We denote the complexifications of $\frak{g}$ and $\frak{a}$ by $\frak{g}^{\mathbb{C}}$ and $\frak{a}^{\mathbb{C}}$, respectively.
Then, $\frak{a}^{\mathbb{C}}$ is a Cartan subalgebra of $\frak{g}^{\mathbb{C}}$.
Let $\tau$ be the complex conjugation of $\frak{g}^{\mathbb{C}}$ corresponding to $\frak{g}$.
Denote by $(\ ,\ )$ the Killing form of $\frak{g}^{\mathbb{C}}$ and set $\langle \ ,\ \rangle = -(\ ,\ )|_{\frak{g} \times \frak{g}}$.
Then, $\langle \ ,\ \rangle$ is a $G$-invariant inner product on $\frak{g}$.
Let $\Sigma$ be the root system of $\frak{g}^{\mathbb{C}}$ with respect to $\frak{a}^{\mathbb{C}}$.
For each $\alpha \in \Sigma$, set $H_{\alpha} \in i\frak{a}$ such that $\alpha(H) = (H_{\alpha}, H)$ for any $H \in \frak{a}^{\mathbb{C}}$.
For any $\alpha, \beta \in \Sigma$, we denote $(H_{\alpha}, H_{\beta})$ by $(\alpha, \beta)$.
Set $A_{\alpha} = (2/(\alpha, \alpha))H_{\alpha}$ for each $\alpha \in \Sigma$.
Take a linear order on $i\frak{a}$ and denote the set of all positive roots by $\Sigma^{+}$.
Let $\beta \in \Sigma^{+}$ be the highest root.
For each $n \in \mathbb{Z}$, we set $\Sigma_{n} = \{ \alpha \in \Sigma\ ;\ (2(\beta, \alpha)/(\beta, \beta)) = n \}$.
Then, 
\[
\Sigma = \Sigma_{-2} \sqcup \Sigma_{-1} \sqcup \Sigma_{0} \sqcup \Sigma_{1} \sqcup \Sigma_{2}, \quad\quad \Sigma_{\pm 2} = \{ \pm \beta \}.
\]
Define $\theta = \mathrm{exp}\pi(iA_{\beta}) \in G$.
Then, $\theta$ is an involutive element of $G$, that is $\theta^{2} = e$, where $e$ is the unit element of $G$.
Denote by the same symbol the inner automorphism of $G$ induced by $\theta$.
Moreover, we denote by the same symbol the induced automorphism of $\frak{g}$ by $\theta$.
Set $\frak{k} = \{ X \in \frak{g}\ ;\ \theta(X) = X \}$ and $\frak{m} = \{ X \in \frak{g}\ ;\ \theta(X) = -X \}$.
Then, $\frak{g} = \frak{k} + \frak{m}$.
Let $\pi_{\frak{m}}:\frak{g} \rightarrow \frak{m}$ be the orthogonal projection with respect to $\langle\ ,\ \rangle$.
Set $K = \{ g \in G\ ;\ \theta(g) = g \}$.
The Lie algebra of $K$ is $\frak{k}$.
Since $G$ is simply connected, $K$ is connected.
Moreover, $(G,K)$ is a compact Riemannian symmetric pair, and $M = G/K$ is a compact symmetric space.
Set $o = eK$.
Then, $T_{o}M$ is identified with $\frak{m}$.
The $G$-invariant metric of $M$ induced by $\langle\ ,\ \rangle$ is also denoted by the same symbol.

For each $\alpha \in \Sigma$, define $\tilde{\frak{g}}_{\alpha} = \{ X \in \frak{g}^{\mathbb{C}}\ ;\ [H,X] = \alpha(H)X\ (H \in \frak{a}^{\mathbb{C}}) \}$.
Let $X_{\alpha} \in \tilde{\frak{g}}_{\alpha}$ satisfy the following conditions:

\begin{itemize}

\item[(a)]
$\tau(X_{\alpha}) = -X_{-\alpha}$, 

\item[(b)]
$[X_{\alpha}, X_{-\alpha}] = A_{\alpha}$,

\item[(c)]
for $\alpha, \gamma \in \Sigma\ (\gamma \not= -\alpha)$, if $\alpha + \gamma \in \Sigma$ then $[X_{\alpha}, X_{\gamma}] = N_{\alpha, \gamma}X_{\alpha + \gamma}$, where $N_{\alpha, \gamma} = \pm (p + 1)$ and $p$ is the greatest integer such that $\gamma - p \alpha \in \Sigma$, and if $\alpha + \gamma \not\in \Sigma$ then $[X_{\alpha}, X_{\gamma}] = 0$.

\end{itemize}

For each $\alpha \in \Sigma^{+}$, we define $Z_{\alpha} = X_{\alpha} + \tau(X_{\alpha}) = X_{\alpha} - X_{-\alpha}$ and $W_{\alpha} = i(X_{\alpha} - \tau(X_{\alpha})) = i(X_{\alpha} + X_{-\alpha})$.
Then,
\[
\begin{split}
\frak{g} &= \frak{a} + \sum_{\gamma \in \Sigma^{+}}(\mathbb{R}Z_{\gamma} + \mathbb{R}W_{\gamma}), \\
\frak{k} &= \frak{a} + (\mathbb{R}Z_{\beta} + \mathbb{R}W_{\beta}) + \sum_{\gamma \in \Sigma^{+} \cap \Sigma_{0}}(\mathbb{R}Z_{\gamma} + \mathbb{R}W_{\gamma}), \\
\frak{m} & = \sum_{\gamma \in \Sigma_{1}}(\mathbb{R}Z_{\gamma} + \mathbb{R}W_{\gamma}).
\end{split}
\]
Set $\frak{s} = \mathbb{R}(iA_{\beta}) + \mathbb{R}Z_{\beta} + \mathbb{R}W_{\beta}$.
Then, $\frak{s}$ is isomorphic to $\frak{sp}(1)$ and forms a $3$-dimensional ideal of $\frak{k}$.
In particular, $\mathrm{Ad}(k)(\frak{s}) \subset \frak{s}$ for any $k \in K$.
Set the equivalence relation $\sim$ on $G \times \frak{s}$ such that $(g_{1}, X_{1}) \sim (g_{2}, X_{2})$ if and only if $g_{2}^{-1}g_{1} \in K$ and $\mathrm{Ad}(g_{2}^{-1}g_{1})X_{1} = X_{2}$.
Denote the quotient space $G \times \frak{s}/\sim$ by $G \times_{K} \frak{s}$.
Moreover, the equivalence class of $(g, X) \in G \times \frak{s}$ is denoted by $[(g,X)]$.
We associate to $[(g,X)] \in G \times_{K} \frak{s}$ the endomorphism of $T_{g(o)}M$ given by $g \circ \mathrm{ad}(X)|_{\frak{m}} \circ g^{-1}$.
This correspondence defines a quaternionic structure $Q$ on $M$ and $(M, Q, \langle\ ,\ \rangle)$ is a quaternionic K\"{a}hler manifold.
Define $S(\frak{s}) = \{ a(iA_{\beta}) + bZ_{\beta} + cW_{\beta}\ ;\ a,b,c \in \mathbb{R}, a^{2} + b^{2} + c^{2} = 1\}$.
Since $\mathrm{Ad}(k)(S(\frak{s})) \subset S(\frak{s})$ for any $k \in K$, we can consider $G \times_{K} S(\frak{s})$.
Then $G \times_{K} S(\frak{s})$ can be identified with the twistor space $Z$ of $M$.
We consider the $G$-action on $G \times_{K} S(\frak{s})$ defined by $G \times (G \times_{K} S(\frak{s})) \rightarrow G \times_{K} S(\frak{s})\ ;\ (g, [(h,X)]) \mapsto [(gh, X)]$.
Since the $K$-action $K \times S(\frak{s}) \rightarrow S(\frak{s})\ ;\ (k,X) \mapsto \mathrm{Ad}(k)X$ is transitive, the $G$-action on $G \times_{K} S(\frak{s})$ is also transitive.
Moreover, $G \times_{K} S(\frak{s}) \ni [(g,X)] \mapsto \mathrm{Ad}(g)(X) \in \mathrm{Ad}(G)(iA_{\beta})$ is a $G$-equivariant diffeomorphism between $G \times_{K} S(\frak{s})$ and $\mathrm{Ad}(G)(iA_{\beta})$.
Thus, we can identify the twistor space $Z$ with $\mathrm{Ad}(G)(iA_{\beta})$.
Finally, we introduce some notations.
For any $1 \leq k \leq n$,
\[
\begin{split}
G^{o}_{k}(\mathbb{R}^{n}) & = SO(n)/SO(k) \times SO(n-k), \\
G_{k}(\mathbb{C}^{n}) &= SU(n)/S(U(k) \times U(n-k)), \\
G_{k}(\mathbb{H}^{n}) &= Sp(n)/Sp(k) \times Sp(n-k). \\
\end{split}
\]
Table 1 lists all compact quaternionic K\"{a}hler symmetric spaces $M = G/K$.

\begin{table}[htbp]

\centering
\begin{tabular}{|c|c|c|c|} \hline

$M$ &$G$ & $K$ & $\dim M$ \\ \hline\hline
$G_{1}(\mathbb{H}^{n})\ (n \geq 3)$ & $Sp(n)$ & $Sp(1) \times Sp(n-1)$ & $4(n-1)$ \\ \hline 
$G_{2}(\mathbb{C}^{n})\ (n \geq 4)$ & $SU(n)$ & $S(U(2) \times U(n-2))$ & $4(n-2)$ \\ \hline 
$G^{o}_{4}(\mathbb{R}^{n})\ (n \geq 7)$ & $Spin(n)$ & $Spin(4) \cdot Spin(n-4))$ & $4(n-4)$ \\ \hline

$G$ & $G_{2}$ & $SO(4)$ & 8  \\ \hline 
$FI$ & $F_{4}$ & $Sp(1) \cdot Sp(3)$ & 28  \\ \hline 
$EII$ & $E_{6}$ & $Sp(1) \cdot SU(6)$ & 40 \\ \hline
$EVI$ & $E_{7}$ & $Sp(1) \cdot Spin(12)$ & 64 \\ \hline
$EIX$ & $E_{8}$ & $Sp(1) \cdot E_{7}$ & 112 \\ \hline

\end{tabular}
\caption{compact quaternionic K\"{a}hler symmetric spaces}
\end{table}


\section{Construction}

In this section, we construct a non-compact totally complex submanifold $N$ of maximal dimension of $M$.
Throughout this section, we assume $G \not= Sp(n)$, that is, $M \not= \mathbb{H}P^{n}$.
Let $\delta \in \Sigma_{1}$ such that $(\delta, \delta) = (\beta, \beta)$ (note that there are no such $\delta$ when $G = Sp(n)$).
Since the curve $b(t) = \mathrm{exp}(tZ_{\delta})\ (0 \leq t \leq 2\pi)$ is a closed one-parameter subgroup of $G$ and $Z_{\delta} \in \frak{m}$, the curve $a(t) = (\mathrm{exp}(tZ_{\delta}))(o)\ (0 \leq t \leq \pi)$ is a closed geodesic of $M$.
Moreover, since $b(t)\ (0 \leq t \leq 2\pi)$ is a shortest closed one-parameter subgroup with respect to the invariant metric defined by the inner product $\langle \ ,\ \rangle$ on $\frak{g}$, it follows that $a(t)\ (0 \leq t \leq \pi)$ is a shortest closed geodesic (\cite{Helgason}, CHAPTER VII, Section 11).
Set $p = a(\pi/2)$.
A maximal totally geodesic sphere with maximal sectional curvature is called a Helgason sphere.
By \cite{Helgason}, CHAPTER VII, Section 11, Theorem 11.1, the dimension of a Helgason sphere is $2$ for any compact quaternionic K\"{a}hler symmetric space except for  quaternionic projective spaces.
Define $\overline{M}_{p} = \mathrm{exp}(\mathbb{R}Z_{\delta} + \mathbb{R}W_{\delta})(o)$.
Then, $\overline{M}_{p}$ is a Helgason sphere through both $o$ and $p$.
Set $M_{p} = \overline{M}_{p} \backslash \{ o \}$.
Let $2i(M)$ be the length of any great circle of a Helgason sphere.

\begin{thm}\cite{Yang} \label{radius}
The injectivity radius of $M$ is $i(M)$.

\end{thm}

The connected component of the fixed point set of a geodesic symmetry $s_{o}$ at $o \in M$ is called a polar of $o$ \cite{Chen-Nagano}.
A polar is a totally geodesic submanifold. 
Denote the polar through $p$ by $M_{o}^{+}(p)$.
Then, $M_{o}^{+}(p)$ coincides with the $K$-orbit $K(p)$.
In particular, $M_{o}^{+}(p)$ is a totally geodesic (locally) totally complex submanifold of maximal dimension of $M$ \cite{Takeuchi}.

Let $\frak{h}$ be the orthogonal complement of $\frak{s}$ in $\frak{k}$ and $\frak{a}_{\beta}$ the orthogonal complement of $\mathbb{R}(iA_{\beta})$ in $\frak{a}$.
Then, we have
\[
\frak{h} = \frak{a}_{\beta} + \sum_{\gamma \in \Sigma^{+} \cap \Sigma_{0}}(\mathbb{R}Z_{\gamma} + \mathbb{R}W_{\gamma}).
\]
Let $H$ be the connected subgroup of $K$ whose Lie algebra is $\frak{h}$.
Then, $H$ is compact and either $K = Sp(1) \times H$ or $K = Sp(1) \cdot H$.
Consider the $H$-orbit $H(p)$.
Define $H_{p} = \{ h \in H\ ;\ h(p) = p \}$ and let $\frak{h}_{p}$ be the Lie algebra of $H_{p}$.
Then, we have $H(p) = H/H_{p}$.
Set $\theta_{p} = \big( \mathrm{exp}(\pi/2)Z_{\delta} \big) \theta \big( \mathrm{exp}(-(\pi/2)Z_{\delta}) \big) \in G$.
Since
\[
\begin{split}
\Big( \mathrm{exp}\frac{\pi}{2}Z_{\delta} \Big) \Big( \mathrm{exp}\pi(iA_{\beta}) \Big) \Big( \mathrm{exp}(-\frac{\pi}{2}Z_{\delta}) \Big) 
&= \mathrm{exp}(\pi iA_{\beta - \delta}),
\end{split}
\]
both $\frak{k}$ and $\frak{s}$ are invariant under $\theta_{p}$.
Consequently, we obtain $\theta_{p}(\frak{h}) \subset \frak{h}$ and $\theta_{p}(H) \subset H$.
Since the isotropy subgroup of $G$ at $p$ is given by $\{ g \in G\ ;\ \theta_{p}(g) = g\}$, we obtain $H_{p} = \{ h \in H\ ;\ \theta_{p}(h) = h \}$ and $\frak{h}_{p} = \{ X \in \frak{h}\ ;\ \theta_{p}(X) = X \}$.
Hence, $(H,H_{p})$ is a compact Riemannian symmetric pair, and $H(p)$ is a totally geodesic submanifold of $M$.
For each $M$, we list $K(p)$ and $H(p)$ in Table \ref{list}.
We see that $H_{p}$ is connected and $H(p)$ is a Hermitian symmetric space of compact type.

\begin{table}[htbp]

\small

\centering
\begin{tabular}{|c|c|c|c|c|c|c|c|c|c} \hline

$M$ & $K(p)$ & $H(p)$  \\ \hline
$G_{2}(\mathbb{C}^{n})\ (n \geq 4)$ & $\mathbb{C}P^{1} \times \mathbb{C}P^{n-3}$ & $\mathbb{C}P^{n-3}$  \\ \hline 
$G^{o}_{4}(\mathbb{R}^{7})$ & $\big( (S^{2} \times S^{2}) \times S^{2} \big)/\mathbb{Z}_{2}$ & $S^{2} \times S^{2}$  \\ \hline
$G^{o}_{4}(\mathbb{R}^{n})\ (n \geq 8)$ & $\big( (S^{2} \times S^{2}) \times G^{o}_{2}(\mathbb{R}^{n-4}) \big) /\mathbb{Z}_{2}$ & $S^{2} \times G^{o}_{2}(\mathbb{R}^{n-4})$  \\ \hline

$G$ & $(S^{2} \times S^{2})/\mathbb{Z}_{2}$ & $S^{2}$ \\ \hline
$FI$ & $(S^{2} \times Sp(3)/U(3))/\mathbb{Z}_{2}$ & $Sp(3)/U(3)$  \\ \hline
$EII$ & $(S^{2} \times G_{3}(\mathbb{C}^{6}))/\mathbb{Z}_{2}$ & $G_{3}(\mathbb{C}^{6})$ \\ \hline 
$EVI$ & $(S^{2} \times SO(12)/U(6))/\mathbb{Z}_{2}$ & $SO(12)/U(6)$  \\ \hline
$EIX$ & $(S^{2} \times E_{7}/(U(1) \cdot E_{6}))/\mathbb{Z}_{2}$ & $E_{7}/(U(1) \cdot E_{6})$  \\ \hline

\end{tabular}
\caption{$K(p)$ and $H(p)$ for each compact quaternionic K\"{a}hler symmetric space}
\label{list}
\end{table}

\normalsize

For $\gamma \in \Sigma$ and $n \in \mathbb{Z}$, we set $\Sigma_{\gamma, n} = \{ \alpha \in \Sigma\ ;\ 2(\gamma, \alpha)/(\gamma, \gamma) = n \}$.
Moreover, we set $\Sigma_{\gamma, n}^{+} = \Sigma_{\gamma, n} \cap \Sigma^{+}$.
Note that $\Sigma_{\beta,n} = \Sigma_{n}$ for each $n \in \mathbb{Z}$ and $\Sigma_{\beta,0}^{+} = \Sigma^{+} \cap \Sigma_{0}$.
Since $\frak{h}_{p} = \{ X \in \frak{h}\ ;\ \theta_{p}(X) = X \}$, we have
\[
\frak{h}_{p} = \frak{a}_{\beta} + \sum_{\gamma \in \Sigma^{+}_{\beta, 0} \ \cap \ \Sigma_{\beta - \delta, 0}}(\mathbb{R}Z_{\gamma} + \mathbb{R}W_{\gamma}).
\]

\begin{lemm}\label{rotation}
$h(\overline{M}_{p}) \subset \overline{M}_{p}$ for any $h \in H_{p}$.

\end{lemm}

\begin{proof}
Since $H_{p}$ is connected and $H \subset K$, it suffices to show that $\mathrm{ad}(T)(\mathbb{R}Z_{\delta} + \mathbb{R}W_{\delta}) \subset \mathbb{R}Z_{\delta} + \mathbb{R}W_{\delta}$ for any $T \in \frak{h}_{p}$.
Suppose that $T \in \frak{a}_{\beta}$.
Then, $[T, Z_{\delta}] \in \mathbb{R}W_{\delta}$ and $[T, W_{\delta}] \in \mathbb{R}Z_{\delta}$.
Let $Z_{\gamma}\ (\gamma \in \Sigma_{\beta, 0}^{+} \cap \Sigma_{\beta- \delta, 0})$.
Since $(\gamma, \beta) = 0$ and $(\gamma, \beta - \delta) = 0$, we have $(\gamma, \delta) = 0$.
Thus, $\gamma \pm \delta \not\in \Sigma$, and $[Z_{\gamma}, Z_{\delta}] = [Z_{\gamma}, W_{\delta}] = 0$.
By a similar argument, we also obtain $[W_{\gamma}, Z_{\delta}] = [W_{\gamma}, W_{\delta}] = 0$.
\qed
\end{proof}

Since $H_{p}$ fixes $o$, we have $H_{p}(M_{p}) \subset M_{p}$.
Note that $M_{p}$ is diffeomorphic to $\mathbb{R}^{2}$ since $\overline{M}_{p}$ is a $2$-dimensional sphere.
The action of $H_{p}$ on $M_{p}$ is by rotations and the fixed point set is $\{ p \}$.
Consider the associated bundle $H \times_{H_{p}} M_{p}$.
Then, $H \times_{H_{p}} M_{p}$ is a rank $2$ vector bundle over $H(p)$.
Define $N = \{ h(a)\ ;\ h \in H, a \in M_{p} \}$ and let $\pi:H \times M_{p} \rightarrow N$ be the map given by $\pi(h,a) = h(a)$ for any $h \in H$ and $a \in M_{p}$.

\begin{lemm} \label{fiber}
Let $(h_{1}, a_{1})$ and $(h_{2}, a_{2})$ be elements of $H \times M_{p}$.
Then, $\pi(h_{1}, a_{1}) = \pi(h_{2}, a_{2})$ if and only if $[(h_{1}, a_{1})] = [(h_{2}, a_{2})] \in H \times_{H_{p}} M_{p}$.

\end{lemm}

\begin{proof}
Assume that $[(h_{1}, a_{1})] = [(h_{2}, a_{2})]$.
Then, by definition, $h_{2}^{-1}h_{1} \in H_{p}$ and $(h_{2}^{-1}h_{1})(a_{1}) = a_{2}$.
Thus, we obtain $h_{1}(a_{1}) = h_{2}(a_{2})$ and $\pi(h_{1}, a_{1}) = \pi(h_{2}, a_{2})$.
Conversely, suppose that $\pi(h_{1}, a_{1}) = \pi(h_{2}, a_{2})$.
It suffices to show that $h_{2}^{-1}h_{1} \in H_{p}$.
Let $k_{i} \in H_{p}$ and $t_{i} \in [ -\pi/2, \pi/2 ] \backslash \{ 0 \}$ for $i = 1,2$ be $a_{1} = k_{1} \mathrm{exp}(t_{1}Z_{\delta})(o)$ and $a_{2} = k_{2}\mathrm{exp}(t_{2}Z_{\delta})(o)$.
Set $k = k_{2}^{-1}h_{2}^{-1}h_{1}k_{1} \in H$.
Since $h_{1}(a_{1}) = h_{2}(a_{2})$, we have
\[
k \big( \mathrm{exp}(t_{1}Z_{\delta}) \big)(o) = \mathrm{exp}(t_{2}Z_{\delta})(o)
\]
and $t_{2} = \pm t_{1}$.
If $t_{1} = \pm \pi/2$, then $k(p) = p$ and $k \in H_{p}$.
Hence, $h_{2}^{-1}h_{1} \in H_{p}$.
On the other hand, if $t_{1} \not= \pm \pi/2$, then $t_{1}\mathrm{Ad}(k)Z_{\delta} = t_{2}Z_{\delta}$ by Theorem \ref{radius}.
Hence, $\mathrm{Ad}(k)Z_{\delta} = \pm Z_{\delta}$ and
\[
p =  \mathrm{exp} \Big( \frac{\pi}{2}Z_{\delta} \Big)(o) = \mathrm{exp} \big( \pm \frac{\pi}{2}\mathrm{Ad}(k)Z_{\delta} \big) (o) = k \Big( \mathrm{exp}(\pm \frac{\pi}{2}Z_{\delta}) \Big)(o) = k(p). 
\]
Therefore, we obtain $k \in H_{p}$ and $h_{2}^{-1}h_{1} \in H_{p}$.
\qed
\end{proof}

By Lemma \ref{fiber}, $N \cong H \times_{H_{p}} M_{p}$ as a manifold, and $N$ is a submanifold of $M$.
Moreover, $N$ is a $\mathrm{rank}$ 2 vector bundle over $H(p)$.
The action of $H$ on $N$ is of cohomogeneity one.
For example, $\{ a(t)\ ;\ 0 < t \leq \pi/2 \}$ is an orbit space of this action.
Hence $N = \{ ha(t)\ ;\ h \in H, 0 < t \leq \pi/2 \}$.
By direct calculations, we obtain $2\dim N = \dim M$.

Next, we consider the tangent space of $N$ at each point.
Let $C = \{ a(t)\ ;\ 0 < t < \pi/2 \}$.
First, we assume $0 < t < \pi/2 $.
Since $T_{a(t)}N = T_{a(t)}C + T_{a(t)}H(a(t))$, we have
\[
b(t)^{-1}T_{a(t)}N = \mathbb{R}Z_{\delta} + \pi_{\frak{m}}(\mathrm{Ad}(b(t)^{-1})\frak{h}).
\]
For each $T \in \frak{h}$, we study $\pi_{\frak{m}}(\mathrm{Ad}(b(t)^{-1})T)$.
Define subspaces $\frak{h}_{1}, \cdots, \frak{h}_{4}$ of $\frak{h}$ by
\[
\begin{array}{ll}
\frak{h}_{1} = \frak{a}_{\beta}, & \frak{h}_{2} = \sum_{\gamma \in \Sigma^{+}_{\beta, 0} \cap \Sigma_{\delta, 0}}(\mathbb{R}Z_{\gamma} + \mathbb{R}W_{\gamma}), \\
\frak{h}_{3} = \sum_{\gamma \in \Sigma^{+}_{\beta, 0} \cap \Sigma_{\delta, 1}}(\mathbb{R}Z_{\gamma} + \mathbb{R}W_{\gamma}), & \frak{h}_{4} = \sum_{\gamma \in \Sigma^{+}_{\beta, 0} \cap \Sigma_{\delta, -1}}(\mathbb{R}Z_{\gamma} + \mathbb{R}W_{\gamma}). \\
\end{array}
\]
Then, $\frak{h} = \sum_{i=1}^{4}\frak{h}_{i}$ and $\frak{h}_{p} = \frak{h}_{1} + \frak{h}_{2}$.
By a similar argument to the proof of Lemma \ref{rotation}, we have $\pi_{\frak{m}}(\mathrm{Ad}(b(t)^{-1})\frak{h}_{1}) = \mathbb{R}W_{\delta}$ because $t \not= \pi/2$.
Moreover, $\pi_{\frak{m}}(\mathrm{Ad}(b(t)^{-1})\frak{h}_{2}) = \{ 0 \}$.
On the other hand, $\mathrm{ad}(Z_{\delta})$ gives a complex structure on $\sum_{\gamma \in \Sigma_{\delta,1}}(\mathbb{R}Z_{\gamma} + \mathbb{R}W_{\gamma})$, and $[Z_{\delta}, Z_{\gamma}] = \pm Z_{\delta - \gamma}$ and $[Z_{\delta}, W_{\gamma}] = \pm W_{\delta - \gamma}$ for any $\gamma \in \Sigma_{\delta, 1}$.
Thus, since $t \not= 0$, 
\[
\begin{split}
\pi_{\frak{m}}(\mathrm{Ad}(b(t)^{-1})\frak{h}_{3}) = \frak{h}_{3}, \quad
\pi_{\frak{m}}(\mathrm{Ad}(b(t)^{-1})\frak{h}_{4}) = \frak{h}_{4}. 
\end{split}
\]
Define a subspace $\frak{m}_{N}$ of $\frak{m}$ by
\[
\frak{m}_{N} = \mathbb{R}Z_{\delta} + \mathbb{R}W_{\delta} + \sum_{\gamma \in \Sigma^{+}_{\beta,0} \cap \Sigma_{\delta, 1}}(\mathbb{R}Z_{\delta - \gamma} + \mathbb{R}W_{\delta - \gamma}) + \sum_{\gamma \in \Sigma^{+}_{\beta,0} \cap \Sigma_{\delta, -1}}(\mathbb{R}Z_{\delta + \gamma} + \mathbb{R}W_{\delta + \gamma}). 
\]
We have $b(t)^{-1}T_{a(t)}N = \frak{m}_{N}$.
Next, we assume $t = \pi/2$.
Then, since $T_{p}N = T_{p}M_{p} + T_{p}H(p)$, we obtain $b(\pi/2)^{-1}T_{p}N = \mathbb{R}Z_{\delta} + \mathbb{R}W_{\delta} + \pi_{\frak{m}}(\mathrm{Ad}(b(\pi/2)^{-1})\frak{h})$.
It is clear that $\pi_{\frak{m}}(\mathrm{Ad}(b(\pi/2)^{-1})\frak{h}_{i})$ is the same as the case of $t \not= \pi/2$ for $2 \leq i \leq 4$.
We can easily verify that $\pi_{\frak{m}}(\mathrm{Ad}(b(\pi/2)^{-1}\frak{h}_{1}) = \{ 0 \}$.
Thus, $b(\pi/2)^{-1}T_{p}N = \frak{m}_{N}$.

In the following, we show that $N$ is a totally complex submanifold of $M$.
Define the subsets $\Delta_{+}$ and $\Delta_{-}$ of $\Sigma_{\beta, 1} \cap \Sigma_{\delta, 1}$ as follows:
\[
\Delta_{+} = \{ \delta - \gamma\ ;\ \gamma \in \Sigma^{+}_{\beta, 0} \cap \Sigma_{\delta, 1}\}, \quad\quad
\Delta_{-} = \{ \delta + \gamma\ ;\ \gamma \in \Sigma^{+}_{\beta, 0} \cap \Sigma_{\delta, -1}\}.
\]
Then, $\frak{m}_{N} = (\mathbb{R}Z_{\delta} + \mathbb{R}W_{\delta}) + \sum_{\gamma \in \Delta_{+} \cup \Delta_{-}}(\mathbb{R}Z_{\gamma} + \mathbb{R}W_{\gamma})$.

\begin{lemm} \label{root}
It follows that
\[
\begin{array}{ll}
\Sigma_{\beta, 1} \cap \Sigma_{\delta,1} = \Delta_{+} \sqcup \Delta_{-}, & \Sigma_{\beta, 1} \cap \Sigma_{\delta, 0} = \{ \beta - \gamma\ ;\ \gamma \in \Sigma_{\beta,1} \cap \Sigma_{\delta,1}\}, \\
\Sigma_{\beta, 1} \cap \Sigma_{\delta, 2} = \{ \delta \}, & \Sigma_{\beta, 1} \cap \Sigma_{\delta, -1} = \{ \beta - \delta \}. \\
\end{array}
\]

\end{lemm}

\begin{proof}
It is obvious that $\Delta_{+} \sqcup \Delta_{-} \subset \Sigma_{\beta,1} \cap \Sigma_{\delta,1}$.
Let $\gamma \in \Sigma_{\beta, 1} \cap \Sigma_{\delta,1}$.
Then, either $\delta - \gamma \in \Sigma_{\beta, 0}^{+} \cap \Sigma_{\delta, 1}$ or $\gamma - \delta \in \Sigma_{\beta, 0}^{+} \cap \Sigma_{\delta, -1}$. 
Since $\gamma = \delta - (\delta - \gamma)$ and $\gamma = \delta + (\gamma - \delta)$, we have $\Sigma_{\beta, 1} \cap \Sigma_{\delta,1} \subset \Delta_{+} \cup \Delta_{-}$.
We can easily see $\{ \beta - \gamma\ ;\ \gamma \in \Sigma_{\beta,1} \cap \Sigma_{\delta, 1} \} \subset \Sigma_{\beta,1} \cap \Sigma_{\delta,0}$.
Let $\epsilon \in \Sigma_{\beta, 1} \cap \Sigma_{\delta, 0}$.
Then, $\beta - \epsilon \in \Sigma_{\beta, 1} \cap \Sigma_{\delta, 1}$, so $\Sigma_{\beta, 1} \cap \Sigma_{\delta, 0} \subset \{ \beta - \gamma\ ;\ \gamma \in \Sigma_{\beta,1} \cap \Sigma_{\delta,1}\}$.
Obviously, $\Sigma_{\beta, 1} \cap \Sigma_{\delta, 2} = \{ \delta \}$.
Let $\lambda \in \Sigma_{\beta, 1} \cap \Sigma_{\delta, -1}$.
Then, $\lambda - \beta \in \Sigma_{\delta, -2}$, so $\lambda - \beta = -\delta$ and hence $\lambda = \beta - \delta$.
\qed
\end{proof}

By Lemma \ref{root}, we have 
\[
\frak{m}_{N} = (\mathbb{R}Z_{\delta} + \mathbb{R}W_{\delta}) + \sum_{\gamma \in \Sigma_{\beta, 1} \cap \Sigma_{\delta, 1}}(\mathbb{R}Z_{\gamma} + \mathbb{R}W_{\gamma}).
\] 
Let $\frak{m}_{N}^{\perp}$ be the orthogonal complement of $\frak{m}_{N}$ in $\frak{m}$.
Then, 
\[
\frak{m}^{\perp}_{N} = \mathbb{R}Z_{\beta- \delta} + \mathbb{R}W_{\beta - \delta} +  \sum_{\gamma \in \Sigma_{\beta,1} \cap \Sigma_{\delta,0}}(\mathbb{R}Z_{\gamma} + \mathbb{R}W_{\gamma}).
\]

\begin{lemm}\label{totally comp}
It follows that 
\[
[iA_{\beta}, \frak{m}_{N}] \subset \frak{m}_{N}, \quad 
[Z_{\beta}, \frak{m}_{N}] \perp \frak{m}_{N}, \quad
[W_{\beta}, \frak{m}_{N}] \perp \frak{m}_{N}.
\]

\end{lemm}

\begin{proof}
For each $\gamma \in \Sigma_{\beta, 1}$, it follows that $[iA_{\beta}, Z_{\gamma}] = W_{\gamma}$ and $[iA_{\beta}, W_{\gamma}] = -Z_{\gamma}$, so $[iA_{\beta}, \frak{m}_{N}] \subset \frak{m}_{N}$.
Let $\lambda \in \Sigma_{\beta, 1} \cap \Sigma_{\delta, 1}$.
Then, $[Z_{\beta}, Z_{\lambda}] = \pm Z_{\beta - \lambda}$ and $[Z_{\beta}, W_{\lambda}] = \pm W_{\beta - \lambda}$.
Moreover, $[Z_{\beta}, Z_{\delta}] = \pm Z_{\beta - \delta}$ and $[Z_{\beta}, W_{\delta}] = \pm W_{\beta- \delta}$.
Thus, we obtain $[Z_{\beta}, \frak{m}_{N}] \perp \frak{m}_{N}$.
Similarly, $[W_{\beta}, \frak{m}_{N}] \perp \frak{m}_{N}$.
\qed
\end{proof}

We set $I_{o} \in Z_{o}$ such that $I_{o} = \mathrm{ad}(iA_{\beta})|_{\frak{m}}$.
Let $Sp(1)_{\delta} = \mathrm{exp}(\mathbb{R}(iA_{\delta}) + \mathbb{R}Z_{\delta} + \mathbb{R}W_{\delta})$.
Then, $Sp(1)_{\delta}$ acts on $\overline{M}_{p}$ transitively.
Let $U(1)_{\delta} = \mathrm{exp} (\mathbb{R}(iA_{\delta}))$.
The isotropy subgroup of $Sp(1)_{\delta}$ at $o$ is $U(1)_{\delta}$.
For any $t \in \mathbb{R}$ and $X_{o} \in \frak{m}$,
\[
\begin{split}
& \mathrm{Ad}(\mathrm{exp}t(iA_{\delta})) \circ I_{o} \circ \mathrm{Ad}(\mathrm{exp}(-t)(iA_{\delta})) X_{o} \\ 
=\ & \mathrm{Ad}(\mathrm{exp}t(iA_{\delta}))[ iA_{\beta}, \mathrm{Ad}(\mathrm{exp}(-t)(iA_{\delta})) X_{o}] \\
=\ & [\mathrm{Ad}(\mathrm{exp}t(iA_{\delta}))(iA_{\beta}), X_{o}] \\
=\ & [iA_{\beta}, X_{o}] \\
=\ & I_{o}(X_{o}).
\end{split}
\]
Hence, we can define an $Sp(1)_{\delta}$-invariant section $\overline{I} \in \Gamma(Z|_{\overline{M_{p}}})$, that is, $\overline{I}_{g(o)} = g \circ I \circ g^{-1}$ for any $g \in Sp(1)_{\delta}$.

\begin{lemm}
$\overline{I}$ is $H_{p}$-invariant.

\end{lemm}

\begin{proof}
For each $g \in G$, we identify the tangent space $T_{g(o)}M$ with $\mathrm{Ad}(g)\frak{m}$.
For $h \in G$, the differential $h:T_{g(o)}M \rightarrow T_{hg(o)}M$ corresponds to $\mathrm{Ad}(g)\frak{m} \ni X \mapsto \mathrm{Ad}(h)X \in \mathrm{Ad}(hg)\frak{m}$.
Then, for any $g \in Sp(1)_{\delta}$, the endomorphism $\overline{I}_{g(o)}$ corresponds to $\mathrm{Ad}(g) \circ \mathrm{ad}(iA_{\beta}) \circ \mathrm{Ad}(g^{-1})$.
Let $g_{1} \in Sp(1)_{\delta}$ and $a = g_{1}(o)$.
Fix $h \in H_{p}$.
Then, there exists $g_{2} \in Sp(1)_{\delta}$ such that $hg_{1} = g_{2}h$.
For any $X \in T_{h(a)}M = T_{g_{2}(o)}M = \mathrm{Ad}(g_{2})\frak{m}$,
\[
\begin{split}
\overline{I}_{h(a)}X
&=
\overline{I}_{g_{2}(o)}X 
=
\mathrm{Ad}(g_{2})[iA_{\beta}, \mathrm{Ad}(g_{2}^{-1})X]
=
[\mathrm{Ad}(g_{2})(iA_{\beta}), X] \\
&=
[\mathrm{Ad}(hg_{1}h^{-1})(iA_{\beta}), X] 
=
[\mathrm{Ad}(hg_{1})(iA_{\beta}), X] \\
&=
\mathrm{Ad}(hg_{1})[iA_{\beta}, \mathrm{Ad}(g_{1}^{-1}h^{-1})X] \\
&=
\mathrm{Ad}(h) \circ \mathrm{Ad}(g_{1}) \circ \overline{I}_{o} \circ \mathrm{Ad}(g_{1}^{-1}) \circ \mathrm{Ad}(h^{-1})(X) \\
&=
\mathrm{Ad}(h) \circ \overline{I}_{a} \circ \mathrm{Ad}(h^{-1})(X).
\end{split}
\]
Thus, it follows that $\overline{I}$ is $H_{p}$-invariant.
\qed
\end{proof}

\begin{lemm}\label{totally complex}
For each $x \in M_{p}$, it follows that $T_{x}N$ is invariant under $\overline{I}_{x}$.
Set $Z_{\overline{I}} = \{ J \in Z|_{\overline{M}_{p}}\ ;\ J \circ \overline{I} = - \overline{I} \circ J \}$.
Then, any $J_{x} \in (Z_{\overline{I}})_{x}$ satisfies $J_{x}(T_{x}N) \perp T_{x}N$.

\end{lemm}

\begin{proof}
Let $g \in Sp(1)_{\delta}$ be $a(t) = g(o)$ for $0 < t \leq \pi/2$.
Since $\mathrm{Ad}(k)\frak{m}_{N} \subset \frak{m}_{N}$ for any $k \in U(1)_{\delta}$ and $g^{-1}b(t) \in U(1)_{\delta}$, we have
\[
g^{-1}T_{a(t)}N = g^{-1}b(t)b(t)^{-1}T_{a(t)}N = \mathrm{Ad}(g^{-1}b(t))\frak{m}_{N} = \frak{m}_{N}.
\]
Thus, $\overline{I}_{a(t)}(T_{a(t)}N) \subset T_{a(t)}N$ by Lemma \ref{totally comp}.
Let $h \in H_{p}$ and $q = h(a(t))$.
Since $\overline{I}$ is $H_{p}$-invariant, 
\[
\overline{I}_{q}(T_{q}N) = h \overline{I}_{a(t)} h^{-1} h(T_{a(t)}N) =h \overline{I}_{a(t)}(T_{a(t)}N) \subset T_{q}N.
\]
Hence, the former part of the statement follows.
Since $Z_{\overline{I}}$ is invariant under both $Sp(1)_{\delta}$ and $H_{p}$, the latter part of the statement follows from Lemma \ref{totally comp}.
\qed
\end{proof}

Denote the restriction of $\overline{I}$ to $M_{p}$ by $I$.
Since $\overline{I}$ is $H_{p}$-invariant and $H_{p}$ fixes $o$, $I$ is $H_{p}$-invariant.
We define the $H$-invariant section of $Z|_{N}$ by $I$ which is denoted by the same symbol such that
\[
I_{g(a)} = g \circ I_{a} \circ g^{-1}\ (g \in H, a \in M_{p}).
\]
By Lemma \ref{totally complex}, $(N, I)$ is a totally complex submanifold of maximal dimension of $M$.
We easily see that $I(T_{q}H(p)) \subset T_{q}H(p)$ for any $q \in H(p)$.
Hence, $(H(p), I|_{H(p)})$ is a complex submanifold of $(N,I)$.
Obviously, $(H(p), I|_{H(p)})$ is a Hermitian symmetric space of compact type with the standard complex structure.
As a vector bundle, $N \ni x \mapsto \mathrm{exp}\big( (\pi/2)(iA_{\beta}) \big)(x) \in N$ gives a complex structure on each fiber of $N$.
Thus, $N$ is a complex line bundle over $H(p)$.
The projection from $H \times_{H_{p}} M_{p}$ onto $H(p)$ is given by $H \times_{H_{p}} M_{p} \ni [(h,a)] \mapsto h(p) \in H(p)$.
Hence, the projection $\rho$ from $N$ onto $H(p)$ as a complex line bundle is given by $\rho(h(a)) \mapsto h(p)$ for any $h \in H$ and $a \in M_{p}$.
Then, we can see that $\rho:(N,I) \rightarrow (H(p), I|_{H(p)})$ is holomorphic.
Thus, $N$ is a holomorphic line bundle over $H(p)$.
Since $\overline{M}_{p}$ is a Helgason sphere of $M$, $N$ has totally geodesic fibers and $N$ is a ruled submanifold of $M$.
Hence, we obtain Theorem \ref{main}.

\begin{thm}\label{main}
$N$ is a totally complex submanifold of maximal dimension of $M$ by the section $I$ of $Z|_{N}$.
Moreover, $N$ is a holomorphic line bundle over the Hermitian symmetric space $H(p)$ of compact type and a ruled submanifold of $M$.
The group $H$ acts on $N$ and this action is of cohomogeneity one.
\end{thm}

Next, we consider whether there exists a compact submanifold of the same dimension as $N$ that contains $N$ as an open part.
We assume that such a submanifold $S$ exists.
Then, $o \in S$ and $T_{o}S = \mathrm{span}_{\mathbb{R}}\{ r\mathrm{Ad}(h)Z_{\delta}\ ;\ r \in \mathbb{R},h \in H\}$.
Hence, $\mathrm{Ad}(H)(Z_{\delta})$ is a round sphere in some subspace of $\frak{m}$.
Since $T_{Z_{\delta}}\mathrm{Ad}(H)Z_{\delta} = [\frak{h}, Z_{\delta}]$ and $\frak{m}_{N} = \mathbb{R}Z_{\delta} + [\frak{h}, Z_{\delta}]$, it follows that $\mathrm{Ad}(H)Z_{\delta}$ is a round sphere in $\frak{m}_{N}$.
Let $0 < t < \pi/2$ and $L_{t} = \mathrm{exp}(t\mathrm{Ad}(H)Z_{\delta})(o)$.
Then, $L_{t}$ is an $S^{1}$-bundle over $H(p)$.
Set $U(1)_{\beta} = \mathrm{exp}(\mathbb{R}(iA_{\beta}))$.
Then, $U(1)_{\beta}$ acts on $L_{t}$.
In particular, $U(1)_{\beta}$ acts on each fiber of $L_{t}$ transitively and $U(1)_{\beta} \backslash L_{t} = H(p)$.
On the other hand, $U(1)_{\beta}$ acts on $\mathrm{Ad}(H)(Z_{\delta})$ and $\mathrm{Ad}(H)(Z_{\delta}) \ni X \mapsto (\mathrm{exp}tX)(o) \in L_{t}$ is a $U(1)_{\beta}$-equivariant diffeomorphism.
Since $\mathrm{ad}(iA_{\beta})|_{\frak{m}_{N}}$ gives a complex structure on $\frak{m}_{N}$ and $\mathrm{Ad}(H)Z_{\delta}$ is a round sphere in $\frak{m}_{N}$, we see that $U(1)_{\beta} \backslash \mathrm{Ad}(H)Z_{\delta}$ is a complex projective space.
Thus, $H(p)$ is also a complex projective space.
Hence, if $M$ is neither $G_{2}(\mathbb{C}^{n}) \ (n \geq 4)$ nor $G_{2}/SO(4)$, then such a submanifold $S$ does not exist.
If $M$ is $G_{2}(\mathbb{C}^{n})\ (n \geq 4)$, then it is easy to verify that $N$ is an open part of a totally complex totally geodesic submanifold $\mathbb{C}P^{n-2}$ through $o$.
In particular, $N = \mathbb{C}P^{n-2} \setminus \{ o \}$ and $N$ is an open part of $\mathbb{C}P^{n-2}$.

Let $M = G_{2}/SO(4)$.
We study whether $\mathrm{Ad}(H)Z_{\delta}$ is a round sphere in some subspace of $\frak{m}$.
By the above arguments, if $\mathrm{Ad}(H)Z_{\delta}$ is not a round sphere in any subspaces of $\frak{m}$, then there does not exist a compact submanifold of the same dimension that contains $N$ as an open part.
It is well known that any connected complete totally geodesic submanifold of a round sphere is the intersection of the round sphere with a subspace and vice versa.
We now examine whether $\mathrm{Ad}(H)Z_{\delta}$ is a totally geodesic submanifold of the round sphere in $\frak{m}$ through $Z_{\delta}$.
Let $e_{1}, \cdots, e_{7}$ be an orthonormal basis of $\mathbb{R}^{7}$ with respect to the standard inner product.
Define an endomorphism $G_{ij}\ (1 \leq i,j \leq 7)$ of $\mathbb{R}^{7}$ by
\[
G_{ij} (e_{k} ) =
\left\{
\begin{array}{lllll}
e_{j} & (k = i), \\
-e_{i} & (k = j), \\
0 & (k \not= i,j ).
\end{array}
\right.
\]
Set elements of $\frak{so}(7)$ as follows:
\[
\begin{array}{lrlrlllll}
V_{1}(\lambda, \mu, \nu) = & \lambda G_{23} + \mu G_{45} + \nu G_{67}, & V_{2}(\lambda, \mu, \nu) = & -\lambda G_{13} - \mu G_{46} + \nu G_{57}, \\
V_{3}(\lambda, \mu, \nu) = & \lambda G_{12} + \mu G_{47} + \nu G_{56}, & V_{4}(\lambda, \mu, \nu) = & -\lambda G_{15} + \mu G_{26} - \nu G_{37}, \\
V_{5}(\lambda, \mu, \nu) = & \lambda G_{14} - \mu G_{27} - \nu G_{36}, & V_{6}(\lambda, \mu, \nu) = & -\lambda G_{17} - \mu G_{24} + \nu G_{35}, \\
V_{7}(\lambda, \mu, \nu) = & \lambda G_{16} + \mu G_{25} + \nu G_{34}.
\end{array}
\]
Then, 
\[
\frak{g}_{2} = \mathrm{span}_{\mathbb{R}} \{ V_{i}(\lambda, \mu, \nu) \ |\ 1 \leq i \leq 7, \lambda + \mu + \nu = 0 \}
\]
is a Lie subalgebra of $\frak{so}(7)$ and the connected Lie subgroup of $SO(7)$ whose Lie algebra is $\frak{g}_{2}$ is the exceptional compact Lie group $G_{2}$ \cite{Yokota}.
For each $1 \leq i \leq 7$, we set $V_{i} = \mathrm{span}_{\mathbb{R}} \{ V_{i}(\lambda, \mu, \nu) \ ;\ \lambda, \mu, \nu \in \mathbb{R} \}$.
It is easy to verify that $V_{i} \cap \frak{g}_{2}$ is a maximal abelian subspace.
Set $\frak{a} = V_{1} \cap \frak{g}_{2}$.
For any $\lambda + \mu + \nu = 0$,
\[
\begin{split}
\big[ V_{1}(\lambda, \mu, \nu), \ V_{2}(0,-1,1) & \pm iV_{3}(0,1,-1) \big] \\ 
&= (\mp i)(\mu - \nu) \big( V_{2}(0,-1,1) + iV_{3}(0,1,-1) \big), \\
\big[ V_{1}(\lambda, \mu, \nu), \ V_{4}(0,-1,1) & \pm iV_{5}(0,1,-1) \big] \\
&= (\mp i)(\lambda - \nu) \big( V_{4}(0,-1,1) \pm iV_{5}(0,1,-1) \big), \\
\big[ V_{1}(\lambda, \mu, \nu), \ V_{6}(0,-1,1) & \pm iV_{7}(0,1,-1) \big] \\
&= (\mp i)(\lambda - \nu) \big( V_{4}(0,-1,1) \pm iV_{5}(0,1,-1) \big), \\
\end{split}
\]
\[
\begin{split}
\big[ V_{1}(\lambda, \mu, \nu), \ V_{3}(2,-1,-1) & \pm iV_{2}(2,-1,-1) \big] \\
&= (\mp i)2\lambda \big( V_{3}(2,-1,-1) \pm iV_{2}(2,-1,-1) \big), \\
\big[ V_{1}(\lambda, \mu, \nu), \ V_{5}(2,-1,-1) & \pm iV_{4}(2,-1,-1) \big] \\
&= (\mp i)2\mu \big( V_{5}(2,-1,-1) \pm iV_{4}(2,-1,-1) \big), \\
\big[ V_{1}(\lambda, \mu, \nu), \ V_{7}(2,-1,-1) & \pm iV_{6}(2,-1,-1) \big] \\
&= (\mp i)2\nu \big( V_{7}(2,-1,-1) \pm iV_{6}(2,-1,-1) \big). \\
\end{split}
\]
Hence, the root system of $\frak{g}_{2}^{\mathbb{C}}$ with respect to $\frak{a}^{\mathbb{C}}$ is $\{ (\pm i)(\lambda - \mu), (\pm i)(\lambda - \nu), (\pm i)(\mu - \nu), \pm i 2\lambda, \pm i 2\mu, \pm i 2\nu \}$.
Let $\beta = i(\mu - \nu)$ and $\delta = i(\mu - \lambda)$.
Then, we obtain
\[
\begin{split}
& \mathbb{R}Z_{\delta} = \mathbb{R}V_{6}(0,1,-1), \quad
\frak{k} = \sum_{i=1}^{3} V_{i} \cap \frak{g}_{2}, \quad \\
& \frak{h} = \mathbb{R}V_{1}(2,-1,-1) + \mathbb{R}V_{2}(2,-1,-1) + \mathbb{R}V_{3}(2,-1,-1).
\end{split}
\]
The invariant inner product $\langle \ , \ \rangle$ satisfies 
\[
\langle V_{i}(\lambda, \mu, \nu), \ V_{j}(\lambda', \mu', \nu') \rangle = 
\begin{cases}
8( \lambda\lambda' + \mu\mu' + \nu\nu') & (i = j), \\
0 & (i \not= j).
\end{cases}
\]
Note that $\{ H \in \frak{h} \ ;\ [H, Z_{\delta} ] = 0 \} = \{ 0 \}$.
Define $L = \mathrm{Ad}(H)Z_{\delta}$.
When we regard $Z_{\delta}$ as a point of $L$, we denote $Z_{\delta}$ by $a$.
Since $T_{a}L = [\frak{h}, Z_{\delta}]$ and
\[
\begin{split}
& [V_{1}(2,-1,-1), V_{6}(0,1,-1)] = V_{7}(0, -3, 3), \\
& [V_{2}(2,-1,-1), V_{6}(0,1,-1)] = V_{4}(-2, -1, -1), \\
& [V_{3}(2,-1,-1), V_{6}(0,1,-1)] = V_{5}(2, -1,-1),
\end{split}
\]
we obtain
\[
T_{a}L = \mathbb{R}V_{4}(2,-1,-1) + \mathbb{R}V_{5}(2,-1,-1) + \mathbb{R}V_{7}(0,1,-1).
\]
Let $N_{a}M$ be the normal space of $T_{a}L$ at $a$ in the round sphere in $\frak{m}$ through $a$.
Then,
\[
N_{a}L = \mathbb{R}V_{4}(0,1,-1) + \mathbb{R}V_{5}(0,1,-1) + \mathbb{R}V_{6}(2,-1,-1) + \mathbb{R}V_{7}(2,-1,-1).
\]
Let $\pi : \frak{m} \rightarrow N_{a}L$ be the orthogonal projection.
For any $X \in \frak{h}$, the Killing vector field on $L$ induced by $X$ is denoted by $X^{*}$.
Moreover, the second fundamental form of $L$ in the round sphere through $a$ in $\frak{m}$ is denoted by $h$.
For any $X, Y \in \frak{h}$,
\[
\begin{split}
h_{a}(X^{*}, Y^{*}) 
&= \pi \Big( \left. \frac{d}{ds}\frac{d}{dt} \mathrm{Ad}(\mathrm{exp}sX) \mathrm{Ad}(\mathrm{exp}tY)Z_{\delta} \right|_{s = t = 0}\Big) = \pi \big( \big[ X, [ Y, Z_{\delta} ] \big] \big).
\end{split}
\]
By direct calculations, we obtain
\[
\begin{split}
[ V_{1}(2,-1,-1), V_{7}(0,1,-1) ] &= V_{6}(0,-3,3), \\
[ V_{2}(2,-1,-1), V_{7}(0,1,-1) ] &= V_{5}(-2,1,1), \\
[ V_{3}(2,-1,-1), V_{7}(0,1,-1) ] &= V_{4}(2,-1,-1), \\ 
[ V_{1}(2,-1,-1), V_{4}(2,-1,-1) ] &= V_{5}(2,-1,-1), \\ 
[ V_{2}(2,-1,-1), V_{4}(2,-1,-1) ] &= V_{6}(-4,1,3), \\ 
[ V_{3}(2,-1,-1), V_{4}(2,-1,-1) ] &= V_{7}(4,-5,1), \\ 
[ V_{1}(2,-1,-1), V_{5}(2,-1,-1) ] &= V_{4}(2,-1,-1), \\ 
[ V_{2}(2,-1,-1), V_{5}(2,-1,-1) ] &= V_{7}(4,1,-5), \\ 
[ V_{3}(2,-1,-1), V_{5}(2,-1,-1) ] &= V_{6}(4,-5,1). \\
\end{split}
\]
Hence, $L$ is not a totally geodesic submanifold of the round sphere through $a$ in $\frak{m}$ and $\mathrm{Ad}(H)Z_{\delta}$ is not a round sphere in any subspace of $\frak{m}$.
Hence, there does not exist a compact submanifold of the same dimension that contains $N$ as an open part.
Summarizing these arguments, we obtain Proposition \ref{H(p)}.

\begin{prop} \label{H(p)}
If a compact quaternionic K\"{a}hler symmetric space $M$ is not $G_{2}(\mathbb{C}^{n})\ (n \geq 4)$, then there does not exist a compact submanifold of the same dimension as $N$ that contains $N$ as an open part.
If $M = G_{2}(\mathbb{C}^{n})\ (n \geq 4)$, then $N$ is an open part of a compact totally complex totally geodesic submanifold $\mathbb{C}P^{n-2}$.
In particular, $N = \mathbb{C}P^{n-2} \setminus \{ o \}$.

\end{prop}

\end{document}